\documentclass[12pt]{amsart} 
\usepackage{amsmath,amsthm,amsfonts,amssymb,latexsym}

\begin{document}

\theoremstyle{plain}

\newtheorem{thm}{Theorem}%[section]
\newtheorem{lem}[thm]{Lemma}
\newtheorem{pro}[thm]{Proposition}
\newtheorem{cor}[thm]{Corollary}
\newtheorem{que}[thm]{Question}
\theoremstyle{definition}\newtheorem{rem}[thm]{Remark}
\theoremstyle{definition}\newtheorem{rems}[thm]{Remarks}
\theoremstyle{definition}\newtheorem{defi}[thm]{Definition}
\theoremstyle{definition}\newtheorem{Question}[thm]{Question}
\newtheorem{con}[thm]{Conjecture}

\newtheorem*{thmA}{Theorem A}
\newtheorem*{thmB}{Theorem B}
\newtheorem*{thmC}{Theorem C}

\newtheorem*{thmAcl}{Main Theorem$^{*}$}
\newtheorem*{thmBcl}{Theorem B$^{*}$}

\newcommand{\Maxn}{\operatorname{Max_{\textbf{N}}}}
\newcommand{\Syl}{\operatorname{Syl}}
\newcommand{\dl}{\operatorname{dl}}
\newcommand{\Con}{\operatorname{Con}}
\newcommand{\cl}{\operatorname{cl}}
\newcommand{\Stab}{\operatorname{Stab}}
\newcommand{\Aut}{\operatorname{Aut}}
\newcommand{\Ker}{\operatorname{Ker}}
\newcommand{\fl}{\operatorname{fl}}
\newcommand{\Irr}{\operatorname{Irr}}
\newcommand{\IBr}{\operatorname{IBr}}
\newcommand{\SL}{\operatorname{SL}}
\newcommand{\NN}{\mathbb{N}}
\newcommand{\N}{\mathbf{N}}
\newcommand{\C}{\mathbf{C}}
\newcommand{\OO}{\mathbf{O}}
\newcommand{\F}{\mathbf{F}}
\newcommand{\bk}{\mathbf{k}}
\newcommand{\FF}{\mathbb{F}}
\newcommand{\CF}{\mathcal{F}}
\newcommand{\CO}{\mathcal{O}}

\renewcommand{\labelenumi}{\upshape (\roman{enumi})}

\newcommand{\PSL}{\operatorname{PSL}}
\newcommand{\PSU}{\operatorname{PSU}}

\providecommand{\V}{\mathrm{V}}
\providecommand{\E}{\mathrm{E}}
\providecommand{\ir}{\mathrm{Irr_{rv}}}
\providecommand{\Irrr}{\mathrm{Irr_{rv}}}
\providecommand{\re}{\mathrm{Re}}

\def\Z{{\mathbb Z}}
\def\C{{\mathbb C}}
\def\Q{{\mathbb Q}}
\def\irr#1{{\rm Irr}(#1)}
\def\ibr#1{{\rm IBr}(#1)}
\def\irrv#1{{\rm Irr}_{\rm rv}(#1)}
\def \c#1{{\cal #1}}
\def\cent#1#2{{\bf C}_{#1}(#2)}
\def\syl#1#2{{\rm Syl}_#1(#2)}
\def\nor{\triangleleft\,}
\def\oh#1#2{{\bf O}_{#1}(#2)}
\def\Oh#1#2{{\bf O}^{#1}(#2)}
\def\zent#1{{\bf Z}(#1)}
\def\det{\mathrm{det}}
\def\ker#1{{\rm ker}(#1)}
\def\norm#1#2{{\bf N}_{#1}(#2)}
\def\alt#1{{\rm Alt}(#1)}
\def\iitem#1{\goodbreak\par\noindent{\bf #1}}
   \def \mod#1{\, {\rm mod} \, #1 \, }
\def\sbs{\subseteq}

\def\gc{{\bf GC}}
\def\ngc{{non-{\bf GC}}}
\def\ngcs{{non-{\bf GC}$^*$}}
\newcommand{\notd}{{\!\not{|}}}
\def\OG{{\mathcal{O}G}}
\def\foc{\mathfrak{foc}}
\def\hyp{\mathfrak{hyp}}
\def\CB{\mathcal{B}}
\def\CC{\mathcal{C}}
\def\tenO{\otimes_\CO}

\newcommand{\Out}{{\mathrm {Out}}}
\newcommand{\Br}{{\mathrm {Br}}}
\newcommand{\Mult}{{\mathrm {Mult}}}
\newcommand{\Inn}{{\mathrm {Inn}}}
\newcommand{\IBR}{{\mathrm {IBr}}}
\newcommand{\IBRL}{{\mathrm {IBr}}_{\ell}}
\newcommand{\IBRP}{{\mathrm {IBr}}_{p}}
\newcommand{\ord}{{\mathrm {ord}}}
\def\id{\mathop{\mathrm{ id}}\nolimits}
\renewcommand{\Im}{{\mathrm {Im}}}
\newcommand{\Ind}{{\mathrm {Ind}}}
\newcommand{\diag}{{\mathrm {diag}}}
\newcommand{\soc}{{\mathrm {soc}}}
\newcommand{\End}{{\mathrm {End}}}
\newcommand{\sol}{{\mathrm {sol}}}
\newcommand{\Hom}{{\mathrm {Hom}}}
\newcommand{\Mor}{{\mathrm {Mor}}}
\newcommand{\Mat}{{\mathrm {Mat}}}
\def\rank{\mathop{\mathrm{ rank}}\nolimits}
\newcommand{\Tr}{{\mathrm {Tr}}}
\newcommand{\tr}{{\mathrm {tr}}}
\newcommand{\Gal}{{\it Gal}}
\newcommand{\Spec}{{\mathrm {Spec}}}
\newcommand{\ad}{{\mathrm {ad}}}
\newcommand{\Sym}{{\mathrm {Sym}}}
\newcommand{\Char}{{\mathrm {char}}}
\newcommand{\pr}{{\mathrm {pr}}}
\newcommand{\rad}{{\mathrm {rad}}}
\newcommand{\abel}{{\mathrm {abel}}}
\newcommand{\codim}{{\mathrm {codim}}}
\newcommand{\ind}{{\mathrm {ind}}}
\newcommand{\Res}{{\mathrm {Res}}}
\newcommand{\Ann}{{\mathrm {Ann}}}
\newcommand{\Ext}{{\mathrm {Ext}}}
\newcommand{\Alt}{{\mathrm {Alt}}}
\newcommand{\AAA}{{\sf A}}
\newcommand{\SSS}{{\sf S}}
\newcommand{\RR}{{\mathbb R}}
\newcommand{\QQ}{{\mathbb Q}}
\newcommand{\ZZ}{{\mathbb Z}}
\newcommand{\NB}{{\mathbf N}}
\newcommand{\ZB}{{\mathbf Z}}
\newcommand{\EE}{{\mathbb E}}
\newcommand{\PP}{{\mathbb P}}
\newcommand{\GC}{{\mathcal G}}
\newcommand{\HC}{{\mathcal H}}
\newcommand{\GA}{{\mathfrak G}}
\newcommand{\TC}{{\mathcal T}}
\newcommand{\SC}{{\mathcal S}}
\newcommand{\RC}{{\mathcal R}}
\newcommand{\bG}{{ \bf G}}
\newcommand\bH{{\bf H}}
\newcommand{\bL} {{\bf L}}
\newcommand{\bM}{{\bf M}}
\newcommand{\bT}{{\bf T}}
\newcommand{\GCD}{\GC^{*}}
\newcommand{\TCD}{\TC^{*}}
\newcommand{\FD}{F^{*}}
\newcommand{\GD}{G^{*}}
\newcommand{\HD}{H^{*}}
\newcommand{\GCF}{\GC^{F}}
\newcommand{\TCF}{\TC^{F}}
\newcommand{\PCF}{\PC^{F}}
\newcommand{\GCDF}{(\GC^{*})^{F^{*}}}
\newcommand{\RGTT}{R^{\GC}_{\TC}(\theta)}
\newcommand{\RGTA}{R^{\GC}_{\TC}(1)}
\newcommand{\Om}{\Omega}
\newcommand{\eps}{\epsilon}
\newcommand{\al}{\alpha}
\newcommand{\chis}{\chi_{s}}
\newcommand{\sigmad}{\sigma^{*}}
\newcommand{\PA}{\boldsymbol{\alpha}}
\newcommand{\gam}{\gamma}
\newcommand{\lam}{\lambda}
\newcommand{\la}{\langle}
\newcommand{\ra}{\rangle}
\newcommand{\hs}{\hat{s}}
\newcommand{\htt}{\hat{t}}
\newcommand{\tn}{\hspace{0.5mm}^{t}\hspace*{-0.2mm}}
\newcommand{\ta}{\hspace{0.5mm}^{2}\hspace*{-0.2mm}}
\newcommand{\tb}{\hspace{0.5mm}^{3}\hspace*{-0.2mm}}
\def\skipa{\vspace{-1.5mm} & \vspace{-1.5mm} & \vspace{-1.5mm}\\}
\newcommand{\tw}[1]{{}^#1\!}
\renewcommand{\mod}{\bmod \,}

\marginparsep-0.5cm

\renewcommand{\thefootnote}{\fnsymbol{footnote}}
\footnotesep6.5pt

\title
[Dade and Alperin-Mckay conjectures]{Dade's ordinary conjecture implies 
the Alperin-McKay conjecture}

\author[Kessar]{Radha Kessar}
\address{Department of Mathematics, 
City, University of London EC1V 0HB, United Kingdom}
\email{radha.kessar.1@city.ac.uk}

\author[Linckelmann]{Markus Linckelmann}
\address{Department of Mathematics, 
City, University of  London EC1V 0HB, United Kingdom}
\email{markus.linckelmann.1@city.ac.uk}

\thanks{This material is based upon work supported by the National 
Science Foundation under Grant No. DMS-1440140 while the authors were in 
residence at the Mathematical Sciences Research Institute in Berkeley, 
California, during the  Spring  2018 semester. The second author 
acknowledges support from EPSRC grant EP/M02525X/1.}   



\begin{abstract} 
We show that  Dade's ordinary conjecture implies the Alperin-McKay 
conjecture.  
We remark that some of the methods can be used to identify
a canonical height zero character in a nilpotent block.

\end{abstract}

\maketitle

%%%%%%%%%%%%%%%%%%%%%%%%%%%%%%%%%%%%%%%%%%%%%%%%%%%%%%%%%%%%%%%%%%%%%
%\section{Introduction}

Dade proved in \cite{DadeII} that his projective conjecture 
\cite[15.5]{DadeII} implies the Alperin-McKay conjecture. 
Navarro showed in \cite[Theorem 9.27]{Nav18book} that the group version of 
Dade's ordinary conjecture implies the McKay conjecture. We show
here that Dade's ordinary conjecture \cite[6.3]{DadeI} implies the 
Alperin-McKay conjecture. Let $p$ be a prime number.

\begin{thm} \label{Theorem1}
If Dade's ordinary conjecture holds for all $p$-blocks of finite 
groups, then the Alperin-McKay conjecture holds for all $p$-blocks of
finite groups.
\end{thm}

The proof combines arguments from Sambale \cite{Sa18} and formal
properties of chains of subgroups in fusion systems from 
\cite{Licontractible}. Let $(K, \CO, k)$ be a $p$-modular system. 
We assume that $k$ is algebraically  closed, and let $\bar K$  be an 
algebraic closure of  $K$. By a character of a finite group, we will 
mean a $\bar K$-valued character. For a finite group $G$ and a block 
$B$ of $\CO G$, let $\Irr(B)$ denote the set of irreducible characters
of $G$ in the block $B$, and let $\Irr_0(B)$ denote the set of 
irreducible height zero characters of $G$ in $B$. For a central 
$p$-subgroup $Z$ of $G$ and a character $\eta$ of $Z$, let 
$\Irr_0(B|\eta)$ denote the subset of $\Irr_0(B)$ consisting of those 
height zero characters  which cover the character $\eta$. 
The following lemma is implicit in \cite{Sa18}.

\begin{lem}  \label{lem:ker}
Let $P$ be a finite $p$-group, let $\CF$ be a saturated  fusion system on 
$P$  and let $Z \leq Z(\CF)$.  Suppose that $\eta$ is a linear 
character of $P$. There exists a linear character $\hat\eta$  of $P$ such 
that $\hat\eta|_Z =\eta|_Z$ and $\foc(\CF) \leq \Ker (\hat\eta)$.   
\end{lem} 

\begin{proof}  
First  consider the case  that $\eta|_Z$ is faithful.   Then 
$Z \cap [P,P]=1 $.  Hence by  \cite[Lemma~4.3]{DGMP}, 
$\foc(\CF)\cap Z=1 $. The result is now immediate. Now suppose 
$Z_0 =\Ker( \eta|_Z )$ and let $\bar\CF =\CF/Z_0 $. By the previous 
argument, applied to $P/Z_0$ and $\bar F$, there exists a character 
$\hat\eta$ of $P/Z_0 $ such that $\hat\eta|_{Z/Z_0} = \eta|_{Z/Z_0}$ and  
$\foc(\bar \CF)  \leq \Ker (\hat\eta)$.  Denote also by $\hat\eta$ the  
inflation of $\hat\eta$ to  $P$. Then $\hat\eta$ has the required 
properties since  $\foc(\bar \CF) = \foc(\CF) Z_0/Z_0 $.  
\end{proof}

The following result is a special case of a result due to Murai; we 
include a proof for convenience.

\begin{lem}[cf. {\cite[Theorem~4.4]{Mu94}}]  \label{lem:height} 
Let $G$ be a  finite group, $B$ be a block of $\CO G$, and $P$ a defect
group of $B$. Let $Z$ be a central  $p$-subgroup of $G$ and let $\eta $ 
be an irreducible character of $Z$ such that 
$\Irr_0(B|\eta) \ne \emptyset$. Then $\eta$ extends to $P$.
\end{lem}

\begin{proof}  
By  replacing $K$ by a suitable finite extension we may assume that 
$K$ is a splitting field for  all subgroups of $G$. Let $i\in$ $B^P$ 
be a source idempotent of $B$ and let $V$ be a $KG$-module affording 
an element of $\Irr_0(B|\eta)$. Then $n:= \dim_K(iV)$ is prime to $p$. 
Since $i$ commutes with $P$, $iV$ is a $KP$-module via $x\cdot iv= ixv$, 
where $x \in P, v \in V$. Let $\rho : P \to$ $\mathrm{GL}_n(K)$ be a 
corresponding representation and let $ \delta: P \to  K^{\times}$ be 
the  determinantal character of $\rho$. Then $\delta|_Z = \eta^n $.  
The result follows since $n$ is prime to $p$.   
\end{proof} 

\begin{lem} \label{lem:count} 
Let $G$ be a finite group, let $B$ be a block of $\CO G$ with a defect
group $P$, and let $Z$ be a central $p$-subgroup of $G$.  
Then $|\Irr_0(B)|$ equals the product of $|\Irr_0(B|1_Z)| $ with the 
number of distinct linear characters $\eta$ of $Z$ which extend to  $P$. 
\end{lem}

\begin{proof}  
Let $\CF=\CF_{(P,e_P)}(G, B)$ be the fusion system of $B$ with respect 
to a maximal $B$-Brauer pair $(P, e_P)$, and let $\eta$ be a linear 
character of $Z$ which extends to  $P$.  Since $Z \leq Z(\CF)$, by Lemma 
\ref{lem:ker}  there exists a linear character $\hat\eta$  of $P$ such 
that $\hat\eta|_Z =\eta $ and  $\foc(\CF) \leq$ $\Ker(\hat\eta)$.  
By the properties of the Brou\'e-Puig $*$-construction \cite{BrPuloc}, 
\cite{Rob08} the map $\chi \mapsto \hat\eta * \chi$ is a bijection 
between $\Irr_0( B|1_Z)$ and $\Irr_0(B|\eta) $. The result follows by 
Lemma \ref{lem:height}.
\end{proof} 

Slightly strengthening the terminology in \cite{Mu11}, we say that a 
pair $(G,B)$ consisting of a finite group $G$ and a block $B$ of $\OG$ 
is a {\it minimal counterexample to the Alperin-McKay conjecture} if 
$B$ is a counterexample to the Alperin-McKay conjecture and if $G$ is 
such that first $|G:Z(G)|$ is smallest possible and then $|G|$ is
smallest possible.

\begin{pro} \label{pro:min}
Let $(G,B)$ be a minimal counterexample to the Alperin-McKay 
conjecture. Then $O_p(G)=1 $. 
\end{pro}

\begin{proof} 
By a result of Murai \cite{Mu11}, we have that $Z:=O_p(G) $ is 
central in $G$. Let $P$ be a defect group of $B$ and let $C$ be the 
block of $\CO N_G(P)$  in Brauer correspondence with $B$. By Lemma 
\ref{lem:count}, $|\Irr_0(B)|=|\Irr_0(C) |$ if and only if 
$|\Irr_0(\bar B)|=|\Irr_0(\bar C) |$ where $\bar B$ (respectively 
$\bar C$)  is the block of $\CO G/Z $ (respectively $\CO N_G(P)/Z $) 
dominated by $B$ (respectively $C$). The result follows since 
$N_{G/Z}(P/Z) =N_G(P)/Z $ and  $\bar B$ and $\bar C$ are in Brauer 
correspondence.
\end{proof}
 
Let $\CF$ be a saturated fusion system on a finite $p$-group $P$,
and let $\CC$ be a full subcategory of $\CF$ which is upwardly
closed; that is, if $Q$, $R$ are subgroups of $P$ such that
$Q$ belongs to $\CC$ and if $\Hom_\CF(Q,R)$ is nonempty, then
also $R$ belongs to $\CC$.   Drawing upon notation and facts from 
\cite[\S 5]{Licontractible},  $S_\vartriangleleft(\CC)$  is 
the category having as objects nonempty chains $\sigma =$ 
$Q_0<Q_1<\cdots <Q_m$ of subgroups $Q_i$ of $P$ belonging to $\CC$ such 
that $m\geq$ $0$ and $Q_i$ is normal in $Q_m$, for $0\leq i\leq m$. 
Morphisms in $S_\vartriangleleft(\CC)$ are given by certain `obvious' 
commutative diagrams of morphisms in $\CF$; see 
\cite[2.1, 4.1]{Licontractible} for details. With this notation, the  
{\it length of}  of a chain $\sigma$ in $S_\vartriangleleft(\CC)$ is 
the integer $|\sigma|=m$. The chain $\sigma$ is called {\it fully 
normalised} if $Q_0$ is fully $\CF$-normalised and if either $m=0$ or 
the chain $\sigma_{\geq 1}=$ $Q_1<Q_2<\cdots < Q_m$ is fully 
$N_\CF(Q_0)$-normalised. Every chain in $S_\vartriangleleft(\CC)$ is 
isomorphic (in the category $S_\vartriangleleft(\CC)$) to a fully 
normalised chain. There is an involution $n$ on the set of fully 
normalised chains which fixes the chain of length zero $P$ and which 
sends any other fully normalised chain $\sigma$ to a fully normalised 
chain $n(\sigma)$ of length $|\sigma| \pm 1$. This involution is 
defined as follows. If $\sigma=$ $P$, then set
$n(\sigma)=$ $\sigma$. If $\sigma=$ $Q_0<Q_1<\cdots<Q_m$ is a fully
normalised chain different from $P$ such that $Q_m=$ $N_P(\sigma)$, 
then define $\sigma$ by removing the last term $Q_m$; if 
$Q_m<N_P(\sigma)$, then define $\sigma$ by adding $N_P(\sigma)$ as last 
term to the chain. Then $n(\sigma)$ is fully normalised, and 
$n(n(\sigma))=\sigma$. Denote by $[S_\vartriangleleft(\CC)]$ the 
partially ordered set of isomorphism classes of chains in 
$S_\vartriangleleft(\CC)$, and for each chain $\sigma$ by $[\sigma]$ 
its isomorphism class. We have a partition
$$[S_\vartriangleleft(\CC)] = \{[P]\} \cup \CB \cup n(\CB)\ ,$$
where $\CB$ is the set of isomorphism classes of fully normalised
chains $\sigma$ satisfying $|n(\sigma)|=$ $|\sigma|+1$. The following
Lemma is a very special case of a functor cohomological statement
\cite[Theorem 5.11]{Licontractible}.

\begin{lem} \label{lem:normal}
With the notation above, let $f : [S_\vartriangleleft(\CC)]\to\Z$ be a
function on the set of isomorphism classes of chains in 
$S_\vartriangleleft(\CC)$ satisfying $f([\sigma])=$ $f([n(\sigma)])$
for any fully normalised chain $\sigma$ in $S_\vartriangleleft(\CC)$.
Then 
$$\sum_{[\sigma]\in [S_\vartriangleleft(\CC)]}\ (-1)^{|\sigma|} 
f([\sigma]) = f([P])\ .$$
\end{lem}

\begin{proof}
The hypothesis on $f$ implies that the contributions from chains in 
$\CB$ cancel those from chains in $n(\CB)$, whence the result.
\end{proof}

\begin{pro} \label{pro:Dade}
Let $G$ be a finite group such that $O_p(G)=$ $1$, and let $B$ 
be a block of $\OG$ with  nontrivial defect groups. Suppose that Dade's 
ordinary conjecture holds for $B$ and that the Alperin-McKay conjecture 
holds for any block of any proper subgroup of $G$. Then the Alperin-McKay
conjecture holds for the block $B$.
\end{pro}

\begin{proof}
Let $(P,e)$ be a maximal $B$-Brauer pair, and denote by $\CF$ the 
associated fusion system on $P$. For $d$ a positive integer, denote by 
$\bk_d(G,B)$ the number of ordinary irreducible characters in $B$ of 
defect $d$. If $p^d=$ $|P|$, then $\bk_d(G,B)$ is the number of height
zero characters, and if $p^d>|P|$, then $\bk_d(G,B)=0$. 

Let $\CC$ be the full subcategory of $\CF$ consisting of all 
nontrivial subgroups of $P$. We briefly describe the standard
translation process between chains in a fusion system of a block
and the associated chains of Brauer pairs.
The map sending a chain $\sigma=$ $Q_0<Q_1<\cdots< Q_m$ in 
$S_\vartriangleleft(\CC)$ to the unique chain of nontrivial 
$B$-Brauer pairs $\tau=$ $(Q_0,e_0)<(Q_1,e_1)<\cdots < (Q_m,e_m)$ 
contained in $(P,e)$ induces a bijection between isomorphism classes 
of chains in $S_\vartriangleleft(\CC)$ and the set of $G$-conjugacy 
classes of normal chains of nontrivial $B$-Brauer pairs (cf. 
\cite[2.5]{Licontractible}). If
$\sigma$ is fully normalised, then the corresponding chain of 
Brauer pairs $\tau=$ $(Q_0,e_0)<(Q_1,e_1)<\cdots<(Q_m,e_m)$ has the
property that $e_\tau=$ $e_m$ remains a block of $N_G(\tau)$, and
by \cite[5.14]{Licontractible}, $N_P(\sigma)=$ $N_P(\tau)$ is a defect 
group of $e_\tau$ as a block of $N_G(\tau)$.  Denote by $n(\tau)$ the 
chain of Brauer pairs corresponding to $n(\sigma)$. 

Let  $d>0$ such that $p^d=$ $|P|$. Define a function $f$ on 
$S_\vartriangleleft(\CC)$ by setting 
$$f([\sigma])= \bk_d(N_G(\tau),e_\tau)$$
for any fully normalised chain $\sigma$ and corresponding chain $\tau$ 
of Brauer pairs. If $N_P(\sigma)$ is a proper
subgroup of $P$, then $f([\sigma])=0$, and if $N_P(\sigma)=P$, then
$f([\sigma])$ is the number of height zero characters of the block
$e_\tau$ of $N_G(\tau)$. Dade's ordinary conjecture for $B$, 
reformulated here in terms of chains of Brauer pairs, asserts that
$\bk_d(G, B)$ is equal to the alternating sum
$$\sum_{[\sigma]\in S_\vartriangleleft(\CC)}\ (-1)^{|\sigma|} 
f([\sigma])\ .$$
The passage between formulations in terms of normalisers of chains
of Brauer pairs rather than normalisers of chains of $p$-subgroups
is well-known; see e.g. \cite[4.5]{LinBredon}, \cite{Rob04}. 

If $|n(\sigma)|=$ $|\sigma|+1$, then, setting $H=$ $N_G(\tau)$, 
we have $N_G(n(\tau))=$ $N_H(N_P(\tau), e_{n(\tau)})$; that is, 
$(N_P(\tau), e_{n(\tau)})$ is a maximal $(H, e_\tau)$-Brauer pair. 
By the assumptions, the  Alperin-McKay conjecture holds for the
block $e_\tau$ of $H$. 
This translates to the equality $f([\sigma])=$ $f([n(\sigma)])$.
That is, the function $f$ satisfies the hypotheses of Lemma 
\ref{lem:normal}. Thus the above alternating sum is
equal to $f([P])$, which by definition is $\bk_d(N_G(P,e), e)$, and
thus the Alperin-McKay conjecture holds for $B$.
\end{proof}

Theorem \ref{Theorem1} follows now immediately from combining 
Propositions \ref{pro:min} and \ref{pro:Dade}.

\begin{rem}
By work of Dade \cite{Dade80} and Okuyama and Wajima \cite{OkWa},
the Alperin-McKay conjecture holds for blocks of finite $p$-solvable
groups. G. R. Robinson pointed out that Proposition \ref{pro:min}
yields another short proof of this fact.
\end{rem}

\begin{rem}
Let $G$ be a finite group, $B$ a block algebra of $\OG$, $(P,e_P)$ a
maximal $(G,B)$-Brauer pair with associated fusion system $\CF$ on $P$,
and let $Z$ be a central $p$-subgroup of $G$. Let $\eta$ be a linear
character of $Z$, and suppose that $\eta$ extends to a linear character
$\hat\eta$ of $P$ satisfying $\foc(\CF)\leq\Ker(\hat\eta)$.
The proof of Lemma \ref{lem:count} is based on the fact that the
$*$-construction $\chi\mapsto$ $\hat\eta * \chi$ yields a bijection
$\Irr(B|1_Z)\to$ $\Irr(B|\eta)$. There is some slightly more structural
background to this. For $\chi\in$ $\Irr(B)$, denote by $e(\chi)$ the
corresponding central primitive idempotent in $K\tenO B$. Set
$$e_1 = \sum_{\chi\in\Irr_0(B|1_Z)}\ e(\chi)\ , \ \ \ 
e_\eta = \sum_{\chi\in\Irr_0(B|\eta)}\ e(\chi)\ .$$
Identify $B$ to its image in $K\tenO B$. Multiplying $B$ by
the central idempotents $e_1$ and $e_\eta$ in $K\tenO B$ yields
the two $\CO$-free $\CO$-algebra quotients $Be_1$ and $Be_\eta$ of $B$.
By \cite[Theorem 1.1]{Linfoc}, there is an $\CO$-algebra automorphism
$\alpha$ of $B$ which induces the identity on $k\tenO B$ and which acts
on $\Irr(B)$ as the map $\chi\to$ $\hat\eta * \chi$. Thus the extension
of $\alpha$ to $K\tenO B$ sends $e_1$ to $e_\eta$ and hence induces an
$\CO$-algebra isomorphism
$$B e_1 \cong B e_\eta\ .$$
\end{rem}

We conclude this note with an observation regarding canonical height
zero characters in nilpotent blocks, based in part on some of the above
methods.

Let $G$ be a finite group, $B$ a block algebra of $\OG$, $P$ a defect
group of $B$ and $i\in$ $B^P$ a source idempotent of $B$. Denote by
$\CF$ the fusion system of $B$ on $P$ determined by the choice of $i$.
Suppose that $K$ is a splitting field for all subgroups of $G$.
For $V$ a finitely generated $\CO$-free $B$-module, denote by
$$\Delta_{V,P,i} : P\to \CO^\times$$
the map sending $u\in$ $P$ to the determinant of the $\CO$-linear
automorphism of $iV$ induced by the action of $u$ on $V$ (this makes
sense since all elements in $P$ commute with $i$). By standard
properties of determinants, this map depends only on the
$(B^P)^\times$-conjugacy class of $i$ and the isomorphism class of the
$K\tenO B$-module $K\tenO V$. Thus if $V$ affords a character $\chi\in$
$\Irr(B)$, we write $\Delta_{\chi,P,i}$ instead of $\Delta_{V,P,i}$.

\begin{pro} \label{mainpoint}
With the notation above, let $\chi\in\Irr(B)$ and $\eta\in$
$\Irr(P/\foc(P))$. Regard $\eta$ as a linear character of $P$. We have
$$\Delta_{\eta * \chi, P, i} = \eta^{\chi(i)} \Delta_{\chi,P,i}\ .$$
\end{pro}

\begin{proof}
The statement makes sense as the value of $\chi$ on an idempotent is a
positive integer. Let $V$ be an $\CO$-free $\OG$-module affording
$\chi$. By \cite[Theorem 1.1]{Linfoc} there exists an $\CO$-algebra
automorphism $\alpha$ of $B$ such that the module $V^\alpha$ (obtained
from twisting $V$ by $\alpha$) affords $\eta * \chi$ and such that
$\alpha(ui)=$ $\eta(u)ui$ for all $u\in$ $P$. Since in particular
$\alpha(i)=i$, it follows that
$$\Delta_{V^\alpha, P,i}(u)= \Delta_{V.P,i}(\eta(u)u)$$
for all $u\in$ $P$. The result follows as $\rank_\CO(iV)=$ $\chi(i)$.
\end{proof}

Denote by $\Irr'(B)$ the set of all $\chi\in$ $\Irr(B)$ such that
$\Delta_{\chi,P,i}$ is the trivial map (sending all elements in $P$ to
$1$). Set $\Irr'_0(B)=$ $\Irr'(B)\cap \Irr_0(B)$.
The maximal local pointed groups on $B$ are $G$-conjugate. Thus if $P'$
is any other defect group of $B$ and $i'\in$ $B^{P'}$ a source
idempotent, then there exist $g\in$ $G$ and $c\in (B^{P'})^\times$ such
that $P'=$ $gPg^{-1}$ and $i'=$ $cg i g^{-1}c^{-1}$. Therefore the
map $\Delta_{V,P,i}$ is trivial if and only if the map 
$\Delta_{V,P',i'}$ is trivial, and hence the sets $\Irr'(B)$ and 
$\Irr'_0(B)$ are independent of the choice of $P$ and $i$. The 
following is immediate.

\begin{pro}
The sets $\Irr'(B)$ and $\Irr'_0(B)$ are invariant under any automorphism
of $G$ which stabilises $B$.
\end{pro}

The next result shows that if $B$ is nilpotent, then $\Irr'_0(B)$
consists of a single element.

\begin{pro}
Suppose that $B$ is nilpotent. Then $|\Irr'_0(B)|=1$. Moreover, if $p$
is odd, then the unique element of $\Irr'_0(B)$ is the unique
$p$-rational height zero character in $B$.
\end{pro}

\begin{proof}
Let $\chi \in \Irr_0(B)$. Since $i$ is a source idempotent of $B$,
$\chi(i)$ is prime to $p$ (see \cite{PiPu}).
Hence if $\eta$, $\zeta$ are linear characters of $P$, then
$\eta^{\chi(i)} =\zeta^{\chi(i)}$ implies that $\eta =\zeta $. Since $B$
is nilpotent, we have that $\foc(\CF) =[P,P]$ and $|\Irr_0 (B) |=$
$|P: [P,P] |$. Thus, by Proposition \ref{mainpoint},  the map
$\chi \mapsto \Delta_{\chi, P,  i}$ is a  bijection from $\Irr_0(B)$  to
$\Irr(P/[P,P])$. This proves the first assertion.

Suppose that $p$ is odd. Let $\chi_0$  be the  unique $p$-rational
character in $\Irr_0(B)$.  Let $W(k) $ be the ring of Witt vectors in
$\CO$. By the structure theory of nilpotent blocks  (see \cite{Punil})
there exists a $W(k) G$-module $V$ affording $\chi_0$. Since the source
idempotent $i$ can be chosen to be in $W(k)G$, we have that
$\Delta_{\chi, P, i}$ takes values in $W(k)$. Since $p$ is odd, it
follows that the trivial character of $P$ is the unique linear character
of $P$ which takes values in $W(k)$.
\end{proof}

\bigskip

\textbf{Acknowledgements.}  
We thank Gunter Malle for helpful comments  on the first version of the 
paper.

\end{document}